\documentclass[11pt,reqno]{amsart}
\usepackage[a4paper]{geometry}
\usepackage{amsthm,hyperref,lmodern,mathrsfs}
\usepackage[latin1]{inputenc}
\usepackage[T1]{fontenc} 
\usepackage{graphicx}
\usepackage[frenchb,english]{babel}
\usepackage{calc}
\usepackage{amsmath,amsthm}

\usepackage{url}
\usepackage{times, amssymb, amscd, mathrsfs, graphicx, color}  
\usepackage{enumerate}

\newtheorem{theo}{Theorem}[section]

\newtheorem{coro}[theo]{Corollary}
\newtheorem{lem}[theo]{Lemma}

\newtheorem{Rq}[theo]{Remark}

\newcommand{\1}{\mathbf{1}} 
\newcommand{\N}{\mathbb{N}}                                              
\newcommand{\R}{\mathbb{R}}                                              

\newcommand{\p}{\mathbb{P}}  
\newcommand{\E}{\mathbb{E}}                                            
 
\newcommand{\egalloi}{\stackrel{d}{=}}

\newcommand{\nocontentsline}[3]{}
\newcommand{\tocless}[2]{\bgroup\let\addcontentsline=\nocontentsline#1{#2}\egroup}

\title[Wasserstein decay of jump-diffusions]{Wasserstein decay of one dimensional jump-diffusions}

\author{Bertrand \textsc{Cloez}}
\address[B.~Cloez]{Laboratoire d'Analyse et de Math\'ematiques Appliqu\'ees, UPEMLV
  UMR8050, Universit\'e Paris-Est Marne-la-Vall\'ee, France} %
\email{\url{mailto:bertrand.cloez(at)univ-mlv.fr}} %
\urladdr{\url{http://perso-math.univ-mlv.fr/users/cloez.bertrand}}

\date{ Compiled \today}



\begin{document}
\maketitle

\begin{abstract}
This work is devoted to the Lipschitz contraction and the long time behavior of certain Markov processes. These processes diffuse and jump. They can represent some natural phenomena like size of cell or data transmission over the Internet. Using a Feynman-Kac semigroup, we prove a bound in Wasserstein metric. This bound is explicit and optimal in the sense of Wasserstein curvature. This notion of curvature is relatively close to the notion of (coarse) Ricci curvature or spectral gap. Several consequences and examples are developed, including an $L^2$ spectral for general Markov processes, explicit formulas for the integrals of compound Poisson processes with respect to a Brownian motion, quantitative bounds for Kolmogorov-Langevin processes and some total variation bounds for piecewise deterministic Markov processes.
\end{abstract}


{\footnotesize %

\medskip


 \tableofcontents
}

\section{Introduction}

We are interested by a process which moves continuously for some random time and then jumps. It can represent some natural phenomena that can be observed at a great variety of scales. To give just few examples, let us simply mention the modeling of the Transmission Control Protocol (TCP) used for data transmission over the Internet \cite{BCGMZ,TCP,GK,GRZ,lcst,vLLO}, parasite evolution or size of cell in biology \cite{GW,Feller,C11,LP}, reliability and queuing \cite{CD08,D93,L04,RT00}. More precisely, this process $X = (X_t)_{t\geq 0}$ has an interval $E \subset \R$ as state space and its infinitesimal generator is given, for any smooth enough function $f : E \mapsto \R$ by
\begin{equation}
\label{eq:generateur}
\forall x \in E, \ \mathcal{L}f(x) = \sigma (x) f''(x) + g(x) f'(x) + r(x) \left( \int_0^1 f\left(F(x,\theta)\right) - f(x) d\theta \right)
\end{equation}
where $\sigma,r, g$ and $x\mapsto F(x, \cdot)$ are smooth ($C^\infty$ for instance), $\theta \mapsto F(\cdot, \theta)$ is measurable, $r$ and $\sigma$ are non negative. We also assume that this operator generates a non explosive Markov process (see \cite{SHS} for sufficient condition). For instance we can assume that $r$ is lower bounded and there exists $K>0$ such that
$$
|b(x)-b(y)|+|\sigma(x)-\sigma(y)|\leq K|x-y|.
$$
Between the jumps, this process evolves like a diffusion which satisfies the following stochastic differential equation:
\begin{equation*}
dX_t = g(X_t) dt + \sqrt{2 \sigma(X_t)} d B_t
\end{equation*}
where $(B_t)_{t \geq 0}$ is a standard Brownian motion. Then, at a random and inhomogeneous time $T$ verifying
$$
\p\left(T>t \ | \ X_s, \ s\leq t \right) = \exp\left(- \int_0^t r(X_s) ds\right),
$$
it jumps. That is $X_{T} = F(X_{T-}, \Theta )$ where $\Theta$ is a uniform variable on $[0,1]$. Then, this process repeats these steps again. It is a Markov process. It is called an hybrid process in \cite{SHS}. When $r=0$ it is a diffusion and when $\sigma =0$ it is a piecewise deterministic Markov process (PDMP) \cite{D93}. \\
\ \\
Some properties are established in the literature. Finding explicit bounds for the speed of convergence is an interesting question which remains open in the general case. Our aim in this paper is to get quantitative estimates for the convergence to equilibrium of $X$. Using Lyapunov techniques, \cite{CD08,L04,RT00} give some conditions to have a geometric convergence.  Nevertheless, this process is, in general, irreversible and it has infinite support. This makes Lyapunov techniques less efficient for the derivation of quantitative exponential ergodicity. Furthermore, another main difficulties is that entropy methods fails. In general, the invariant measure of the process does not verify a Poincar\'{e} or log-Sobolev inequality (see remark \ref{eq:TCP-BE} and \cite{W10}). In this work, we use a gradient estimate via a Feynman-Kac formula. We obtain a bound in Wasserstein metric. This bound is optimal in the sense of Wasserstein curvature introduced by Joulin and Ollivier. \\
 \ \\
The two next subsections introduce the notion of Wasserstein curvature and states our main results. Section \ref{sect:ppte-Lipschitz} is devoted to the Lipschitz contraction of Markov processes while the quantitative bounds of jumps diffusions are in Section \ref{sect:curv-jump}. In the last section, we develop some examples and applications including a Wasserstein bound for Kolmogorov-Langevin processes with non convex potential, a variation total bound for  piecewise deterministic Markov processes, some explicit formulas for the TCP windows size process and the integral of L\'{e}vy processes with respect to a Brownian motion. \\
 \ \\
\textbf{Lipschitz contraction and Wasserstein curvature.} We present briefly some basic definitions and properties about Lipschitz contraction of Markov semigroups. The Wasserstein distance between two probability measures $\mu_1, \mu_2$ is defined by
\begin{equation*}
  \mathcal{W} (\mu_1 , \mu_2) %
  = \inf_{\nu \in \text{Marg}(\mu_1, \mu_2)}\,\iint_{E\times E} |x -y| \nu(dx,dy) ,
\end{equation*}
where $\text{Marg}(\mu_1, \mu_2)$ is the set of probability measures on $E^2$ with marginal distributions are $\mu_1$ and $\mu_2$, respectively. This infimum is attained \cite[Theorem 1.3]{vil}. The Kantorovich-Rubinstein duality gives \cite[Theorem 5.10]{vil} the following representation:
\begin{equation}
\label{eq:KR}
  \mathcal{W}(\mu_1,\mu_2) =\sup_{g\in \text{Lip}_1} \, \int_E g \ d\mu_1 - \int_E g \ d\mu_2 ,
\end{equation}
where $\text{Lip}_1$ is the set of Lipschitz function $g$ verifying $|g(x) - g(y)| \leq |x-y|$ for any $x,y\in E$. Let $(P_t)_{t\geq 0}$ be the semigroup of $(X_t)_{t\geq0}$. It is defined by $ P_t g(x)= E[g(X_t) | X_0=x]$, for any $g$ smooth enough. The Wasserstein curvature of $(X_t)_{t\geq 0}$ is the optimal (largest) constant $\rho $ in the following contraction inequality:
\begin{equation}
  \label{eq:contract_lip}
   \sup_{\underset{x\neq y}{x,y \in E}} \frac{\vert P_t g(x) - P_t g(y) \vert }{|x-y|} \leq e^{-\rho t },
\end{equation}
for all $g \in \text{Lip}_1 (d)$ and $t\geq0$. It is actually equivalent to 
\begin{equation*}
\mathcal{W} (\mu_1 P_t, \mu_2 P_t) \leq e^{-\rho t } \, \mathcal{W} (\mu_1, \mu_2),
\end{equation*}
for any  probability measures $\mu_1, \mu_2$ and $t\geq 0$. Here $\mu_1 P_t$ stands for the law of $(X_t)_{t\geq0}$ starting from a random variable distributed according to $\mu_1$. That is,
$$
\mu_1 P_t f(x)= \int_E P_t f (x) \mu_1(dx).
$$
This notion of curvature was introduced by Joulin \cite{J07} and Ollivier \cite{O10} and is connected to the notion of Ricci curvature on Riemannian manifolds \cite{SvR}. It is also a generalisation of the Dobrushin's Uniqueness criterion. If the optimal constant $\rho$ is positive, then the process has a stationary probability measure $\pi$ and the semigroup converges exponentially fast in Wasserstein distance $\mathcal{W}$ to the stationary distribution. This result is a direct consequence of \cite[Theorem 5.23]{C04}. In general, $(X_t)_{t\geq0}$ is not ergodic when $\rho\leq 0$ as can be easily checked with Brownian motion. We can notice that an exponential decay is possible even though the Wasserstein curvature is null (see Lemma \ref{th:WC=0}). The Wasserstein curvature is a local characteristic. It takes into account the behavior on all the space and at any time. It corresponds to the worst possible decay. This notion of curvature is connected with the notion of $L^2-$spectral gap:

\begin{theo}[Wassertsein contraction implies $L^2-$spectral gap for reversible semigroup]
\label{th:Intro-spectralgap}
Let $(P_t)_{t \geq 0}$ be the semigroup of a Markov process with invariant distribution $\pi$. If the following assumptions hold
\begin{itemize}
\item the Wasserstein curvature $\rho$ is positive;
\item the semigroup is reversible;
\item $\lim_{t \rightarrow 0} \Vert P_t f - f \Vert_{L^2(\pi)}=0$ for all $f\in L^2(\pi)$;
\item Lip$_1 \cap L^\infty(\pi) \cap L^2(\pi) $ is dense in $L^2(\pi)$;
\end{itemize}
then it admits an invariant distribution $\pi$ and
$$
\text{Var}_\pi (P_t f) \leq e^{-2\rho t} \text{Var}_\pi (f),
$$
where $\text{Var}_\pi (f)= \int_E (f - \int_E f d\pi)^2 d\pi=\Vert f -\int f d\pi \Vert_{L^2(\pi)}$.
\end{theo}

The proof relies crucially on reversibility. Nevertheless, for most of our examples, this condition is never verified (while in contrast all one dimensional diffusions are reversible). This theorem is just the continuous time adaptation of \cite[Proposition 2.8]{HSV} and is close to \cite[Theorem 5.23]{C04}, \cite[Theorem 2.1 (2)]{W03}, \cite[Theorem 2]{V12}, and \cite[Proposition 30]{O10}.

 \ \\
 \textbf{Main results: Quantitative bounds for jump-diffusions.}
Let $(P_t)_{t\geq0}$ be the semigroup, of the Markov process $(X_t)_{t\geq0}$, generated by \eqref{eq:generateur}. By the It\^{o}-Dynkin Theorem,
\begin{equation}
\label{eq:ito}
P_t f(x) = f(x) + \int_0^t P_s \mathcal{L}f(x) ds = f(x) + \int_0^t \mathcal{L} P_s f(x) ds 
\end{equation}
for all $f$ smooth enough. See \cite{SHS,D93} for full details on the domain of $\mathcal{L}$. We can also describe the evolution of $X$ in terms of stochastic differential equation. Let $(B_s)_{s\geq0}$ be a standard Brownian motion, and let $Q(ds,du,d\theta)$ be a Poisson point process on $\R_+ \times \R_+ \times [0,1]$, of intensity $ds \ du \ d\theta$, independent from the Brownian motion. Here $ds, du, d\theta$ are the Lebesgue measures on $\R_+, \R_+$ and $[0,1]$. Then we have,
\begin{align}
\label{eq:EDSsaut}
X_t= X_0 &+ \int_0^t g(X_s) ds + \int_0^t \sqrt{2 \sigma(X_s)} d B_s \\
&+ \int_0^t \int_{E\ \times [0,1]} \1_{\{u \leq r(X_{s-})\}} \left( F(X_{s-},\theta) - X_{s-} \right) Q(ds,du,d\theta). \nonumber
\end{align}
Some sufficient conditions for the existence of a solution can be found in \cite{SHS,D93}. Before expressing our main results, we recall that a Markov process is stochastically monotone when for any $x,y \in E$, if $x \geq y$ then there exists a coupling $(X,Y)$ starting from $(x,y)$ such that $X_t \geq Y_t$ almost surely for all $t\geq0$.

\begin{theo}[Wasserstein curvature for stochastically monotonous jump-diffusion]
\label{th:Intro-FK}
Let $(X_t)_{t\geq0}$ be a solution of \eqref{eq:EDSsaut}, if it is stochastically monotonous and if
$$
\rho = \inf_{x\in E} \left( - g'(x) + r(x) \left( 1 - \int_0^1 \partial_x F(x,\theta) d\theta \right) - r'(x) \int_0^1 x - F(x,\theta) d\theta \right) \geq 0,
$$
then the contraction inequality \eqref{eq:contract_lip} is satisfied with the optimal constant $\rho$; namely
\begin{equation}
\label{eq:expodecay}
\mathcal{W} \left(\mu P_t, \nu P_t\right) \leq e^{-\rho t} \mathcal{W} \left(\mu , \nu \right)
\end{equation}
for all probability measures $\mu,\nu$.
\end{theo}
In particular, the inequality \eqref{eq:expodecay} holds when all these conditions are satisfied:
\begin{itemize}
 \item $r$ is decreasing;
 \item $F(x,\theta) < x$, for almost all $\theta \in [0,1]$;
 \item $x\mapsto F(x,\theta)$ is increasing.
\end{itemize} 
For instance, the jumps can be multiplicative (that is $F(x,\theta) = \varphi(\theta) x$, where $\varphi \leq 1$), additive (that is $F(x,\theta) = x - \varphi(\theta) $, where $\varphi \geq0$), or a combination of these two types. We can, of course, also suppose that $r$ is increasing, and, for almost all $\theta \in [0,1]$, $F(x,\theta) > x$. 

In the expression of $\rho$, we can see the interplay between the drift parameter $g$ and the jump mechanism $r,F$. We also see that the curvature does not depend on the diffusive term $\sigma$. The proof of this result is based on an expression of the gradient using a Feynman-Kac semigroup (see for instance \cite{Kac}). This expression gives some rates of convergence when $\rho \leq 0$. This theorem gives an interesting estimates in Wasserstein distance and is close to \cite[Theorem 2.2]{W10}. A bound in total variation is given in Section \ref{sect:Example}.

\section{Lipschitz contraction of Markov processes}

\label{sect:ppte-Lipschitz}
In this section, we give some properties of the Wasserstein curvature. We also prove Theorem \ref{th:Intro-spectralgap}.

\subsection{Properties of the curvature} 

Let $(P_t)_{t\geq0}$ be a Markov semigroup, we denote by $\rho$ its Wasserstein curvature. Here, we will give two lemmas which give some properties of the curvature. We begin to prove that $(P_t)_{t\geq0}$ can be geometrically ergodic even if its curvature is not positive. The second lemma is a scaling property.

\begin{lem}[Exponential decay when $\rho\leq 0$]
\label{th:WC=0}
 If $\rho\leq0$ and there exists $t_0>0$ such that 
 $$
 K_{t_0}=\sup_{x_0,y_0\in E}\frac{\mathcal{W}(\delta_{x_0} P_{t_0},\delta_{y_0} P_{t_0})}{|x_0 -y_0|} < 1,
 $$ 
then there exists $\kappa > 0$ such that, for all $x_0,y_0 \in E$, we have
$$
 \forall t \geq t_0, \mathcal{W}(\delta_{x_0} P_{t},\delta_{y_0} P_{t}) \leq e^{-t \kappa} |x_0 - y_0|.
$$
Furthermore we can choose
$$
\kappa = - \frac{\ln(K_{t_0})}{t_0}.
$$
\end{lem}

Lemma \ref{th:WC=0} implies the existence of an invariant distribution and the convergence to it. Theorem \ref{th:cv-WC=0} gives an application of this result. Notice that the conditions of Lemma \ref{th:WC=0} are not always verified. For instance, if $(L_t)_{t \geq 0}$ is a L\'{e}vy process then its semigroup $(P_t)_{t\geq0}$ verifies
$$
\mathcal{W}(\delta_x P_t, \delta_y P_t) = |x-y|.
$$
Now, let $\omega$ and $\bar{\omega}$ be defined, for all $x\neq y$, by
$$
\omega(t,x,y)= \frac{\mathcal{W} \left( \delta_x P_t, \delta_y P_t \right)}{|x -y|},
$$
and
$$
\bar{\omega}(t) = \sup_{x,y \in E} \omega(t,x,y).
$$
The proof of the previous lemma is based on the fact that $\ln (\bar{\omega})$ is sub-additive.

\begin{proof}[Proof of Lemma \ref{th:WC=0}]
First, we have
$$
\mathcal{W}(\mu P_t, \nu P_t) \leq \bar{\omega}(t) \mathcal{W}(\mu , \nu).
$$
Then, for any $x,y \in E$ and by Markov property, we have
\begin{align*}
\mathcal{W}(\delta_x P_{t+s}, \delta_y P_{t+s}) 
&= \mathcal{W}((\delta_x P_t) P_s, (\delta_y P_t) P_s) \\
&\leq \bar{\omega}(t) \mathcal{W}(\delta_x P_s , \delta_y P_s).
\end{align*}
We deduce that
\begin{align*}
\omega(t+s,x,y) \leq \bar{\omega}(t) \omega(s,x,y) \ \Rightarrow \ \bar\omega(t+s) \leq \bar{\omega}(t) \bar\omega(s).
\end{align*}
Now, the curvature is non-negative, thus
$$
\forall t>0, \ \bar\omega(t)\leq 1.
$$ 
So $\bar\omega$ is decreasing. But, as there exists $t_0>0$ such that $\bar\omega(t_0)<1$, we have
$$
\forall t \geq t_0, \ \bar\omega(t) < 1.
$$ 
Finally, for all $t\geq t_0$, there exists $n\in \N$ such that $t \geq n t_0$, and then
\begin{align*}
\mathcal{W}(\delta_x P_t, \delta_y P_t) 
&\leq \bar{\omega}(t) \ \mathcal{W}(\delta_x, \delta_y ) \leq \bar{\omega}(n t_0) \ \mathcal{W} (\delta_x, \delta_y ) \\
&\leq \bar{\omega}(t_0)^n \ \mathcal{W} (\delta_x, \delta_y ) \leq \exp\left(t \ \frac{\ln(\bar{\omega}(t_0))}{t_0} \right) \ \mathcal{W} (\delta_x, \delta_y ).  \\
\end{align*}
\end{proof}

If you change the time scale, the new curvature is easily calculable:

\begin{lem}[Markov processes indexed by a subordinator]
Let $(X_t)_{t\geq0}$ be a Markov process with curvature $\rho$ and let $(\tau(t))_{t\geq0}$ be a subordinator independent of $X$. If $(Q_t)_{t\geq0}$ is the following semigroup:
$$
Q_t f(x) = \E[f(X_{\tau(t)}) \ | \ X_0=x],
$$
then its Wasserstein curvature $\rho_Q$ verifies
$$
\rho_Q= b\rho + \int_0^\infty (1-e^{-\rho z}) \nu(dz)=\psi(\rho),
$$
where $\psi$ is the Laplace exponent of $\tau$, $b$ its drift term and $\nu$ its L\'evy measure.
\end{lem} 
The proof is straightforward. Let us end this subsection with a remark about the notion of Wasserstein spectral gap \cite{HSV} which is close to our Wasserstein curvature.

\begin{Rq}[Wasserstein spectral gap]
A semigroup possesses a positive Wasserstein spectral gap \cite{HSV} if there exist $\lambda>0$ and $C>0$ such that for all $x,y\in E$ and $t\geq0$,
$$
\mathcal{W} \left( \delta_x P_t, \delta_y P_t \right) \leq C e^{- \lambda t} d(x,y).
$$
If $C\leq 1$ then $\lambda$ is the Wasserstein curvature. This difference may be important for concentration inequalities \cite{J07,O10}. Almost all our results are generalisable in the case of positive Wasserstein curvature. 
\end{Rq}

\subsection{Proof of Theorem \ref{th:Intro-spectralgap}: Lipschitz contraction implies $L^2-$Spectral Gap}

The proof of this result is the continuous-time adaptation to \cite[Proposition 2.8 ]{HSV}.
\begin{proof}[Proof of Theorem \ref{th:Intro-spectralgap}]
Let $f$ be a non-negative, Lipschitz and  bounded function such that $\int_E f d\pi =1$. Using the reversibility and the invariance of $\pi$, we have,
\begin{align*}
\text{Var}_\pi (P_t f)
&= \int_E f P_{2t} f d\pi - \left( \int_E f d\pi \right)^2\\
&\leq \Vert f \Vert_{\text{Lip}(d)} \mathcal{W}  \left( P_{2t}f d\pi, \pi\right).
\end{align*}
The measure $P_{2t}f d\pi$ verifies, for all smooth $\varphi$,
\begin{align*}
\int_E \varphi P_{2t}f d\pi 
&=\int_E\int_E \varphi(x) f(y) P_{2t}(x,dy) \pi(dx)\\
&=\int_E\int_E \varphi(x) f(y) P_{2t}(y,dx) \pi(dy).\\
\end{align*}
Thus,
\begin{align}
\text{Var}_\pi (P_t f) 
&\leq \Vert f \Vert_{\text{Lip}(d)} \mathcal{W}  \left( (f\pi)P_{2t}, \pi\right) \nonumber \\
&\leq C_f e^{- 2 \rho t}. \label{eq:ineqnonuniform}
\end{align}

By translation and dilatation, the last inequality holds for all Lipschitz and bounded $f$. Now, let $f\in L^2(\pi)$ be a Lipschitz and bounded function such that,
$$
\int_E f(x) \pi(dx)= 0 \ \text{ and } \ \int_E f(x)^2 \pi(dx)= 1.
$$
Applying spectral Theorem, Jensen inequality and \eqref{eq:ineqnonuniform}, we find
\begin{align*}
\text{Var}_\pi (P_t f)
&=\int P_t f^2 d\pi
= \int_E \int_0^\infty e^{- \lambda t} d E_\lambda(f) d\pi\\
&\leq \left( \int_E \int_0^\infty e^{- \lambda (t+s)} d E_\lambda(f) d\pi \right)^{\frac{t}{t + s}}\\
&\leq C_f^{\frac{t}{t + s}} e^{-2\rho t}.
\end{align*}
Taking the limit $s\rightarrow + \infty$, we conclude the proof.
\end{proof}

This result can not be generalised in the non reversible case as can be viewed in the remark \ref{eq:TCP-BE} below.

\begin{Rq}[Another approach]
We can give an alternative proof, using Inequality \eqref{eq:ineqnonuniform} and \cite[Lemma 2.12]{CGZ}. This lemma is based on the convexity of the mapping $t \mapsto \ln \Vert P_t f \Vert_{L^2(\pi)}$ which is also a consequence of the reversibility.
\end{Rq}

\section{Wasserstein decay of jump-diffusions}
\label{sect:curv-jump}
In this section, we will prove Theorem \ref{th:Intro-FK}.

\subsection{Proof of Theorem \ref{th:Intro-FK}: gradient estimate via Feynman-Kac formula}
In this section we follow the approach of \cite{CJ10}. Let $X$ be generated by \eqref{eq:generateur} and be stochastically monotone. Let $(P_t)_{t\geq0}$ be its semigroup. We begin by estimating the derivative of our semigroup. More precisely, we prove, for any smooth enough function $f$, the following formula:
\begin{equation}
\label{eq:intervertion}
\forall x \geq 0, \ (P_t f)'(x) =\E\left[f'(Y_t) e^{- \int_0^t V(Y_s) ds }\ | \ Y_0=x \right].
\end{equation}
Where $Y$ is a Markov process generated by
\begin{align*}
\mathcal{L_S} f(x) &= \sigma(x) f''(x) +  \left( \sigma'(x) + g(x)\right) f'(x) \\
&+  r(x) \int_0^1 \partial_x F(x,\theta) d\theta \left(  \frac{1}{ \int_0^1 \partial_x F(x,\theta) d\theta} \int_0^1\partial_x F(x,\theta) f(F(x,\theta)) d\theta - f(x)\right) \\
& - r'(x) \int_0^1 x - F(x,\theta) d\theta \left( \frac{1}{\int_0^1 x - F(x,\theta) d\theta} \int_0^1 \int_{F(x,\theta)}^{x} f(u) du \ d\theta - f(x) \right).
\end{align*}
Here, we have used the convention $0 \times \frac{1}{0}=0$. The denominators are not null because of the stochastic monotonicity. And also because we can assume, without less of generality, that $F(x,\cdot) \neq x$ almost surely.
\begin{lem}[Intertwining relation and gradient estimate]
\label{lem:Grad-est}
If we have
$$ 
\forall x \in E, \ \E\left[ \exp\left(- \int_0^t V(Y_s) ds \right) \ | \ Y_0=x \right] < + \infty,
$$
where
$$
V(x)= - g'(x) + r(x) \left( 1 - \int_0^1 \partial_x F(x,\theta) d\theta \right) - r'(x) \int_0^1 x - F(x,\theta) d\theta,
$$
then \eqref{eq:intervertion} holds.
\end{lem}

\begin{proof}
As $ \left(\mathcal{L} f \right)'= \left(\mathcal{L_S} - V  \right) f'$, It\^{o}-Dynkin Formula gives 
\begin{align*}
(P_t f)'(x) 
&= f'(x) + \int_0^t (\mathcal{L} P_s f)'(x) ds \\
&= f'(x) + \int_0^t \left(\mathcal{L_S} - V  \right) (P_s f)'(x) ds.
\end{align*}
But, it is known that the following semigroup verifies also the previous equation:
$$
S_t f(x) = \E\left[f(Y_t) e^{-\int_0^t V(Y_s) ds}\right]
$$
where $Y$ is generated by $\mathcal{L}_S$. As there is a unique (weak) solution, we get $(P_t f)'= S_t f'$.
\end{proof}
We deduce that:
$$
f'\geq 0 \Rightarrow P_t f'\geq 0.
$$
\begin{coro}[Propagation of monotonicity]\label{coro:monotoni}
Under the same assumptions, if $f$ is non-increasing then $P_t f$ is also non-increasing.
\end{coro}

This property is known to be equivalent to the stochastic monotonicity. It implies that this Feynman-Kac representation of the gradient is reserved to the monotonous process. We can not prove it in another context. This explain why this approach fails for a general class of parameters (and in particular of a large class of jumps rates). This proof seems also not be generalisable to higher dimension (except if the process have a radial behavior).

\begin{proof}[Proof of Theorem \ref{th:Intro-FK}]

Let $f$  be a Lipschitz and smooth function. By the intertwining identity (\ref{eq:intervertion}) and for any $x,y \geq 0$, we have
\begin{align*}
|P_t f(x) -P_t f(y)| 
&\leq \sup_{z\geq 0} |(P_t f)'(z)| |x-y| \\
&\leq \sup_{z\geq 0} \E\left[|f'(Y_t)| e^{-\int_0^t V(Y_s) ds} | Y_0 = z\right] |x-y| \\
&\leq \sup_{z\geq 0} \E\left[ e^{-\int_0^t V(Y_s) ds} | Y_0 = z\right] |x-y|.
\end{align*}
So that dividing by $|x-y|$ and taking suprema entail the following inequality
$$
    \sup_{\underset{x\neq y}{x,y \in E}} \frac{\vert P_t f(x) - P_t f(y) \vert }{|x-y|}  \, \leq \, \sup_{z\geq 0} \, \E \left[ \exp \left( - \int
       _0^t V (Y_s) \, ds \right) \ | \ Y_0= z \right].
       $$
Finally, taking $f(x)=x$, we show that the supremum is attained. It achieves the proof. 
\end{proof}

\begin{Rq}[$h-$transform and first eigenvalue]
\label{rq:h-transform}
Assume that $\mathcal{L}_S-V$ has a first eigenvalue $\lambda>0$ such that its eigenvector $\psi$ is positive. Using an $h-$transform with the space-time harmonic function $h=e^{- \lambda t} \psi$, we get for any Lipschitz function $f$, 
$$
\E\left[f'(Y_t)e^{-\int_0^t V(Y_s) ds }\right] = e^{-\lambda t} \psi(x) \E\left[ \frac{f'(Z_t)}{\psi(Z_t)} \right] \leq  e^{-\lambda t} \E\left[ \frac{\psi(Z_0)}{\psi(Z_t)} \right],
$$
where $(Z_t)_{t\geq0}$ is another Markov process. Then, if $\psi$ is smooth enough, the Wasserstein decay is in order to $e^{-\lambda t}$. Section \ref{sect:Langevin} gives an application for Kolmogorov-Langevin processes with non convex potential.
\end{Rq}

\begin{Rq}[A proof by coupling]
The stochastically monotony can be described in coupling term. Indeed, for all $x<y$, there exists a coupling $(X^x,X^y)$, such that the marginals are generated by $\mathcal{L}$, which start from $(x,y)$, and
$$
\forall t \geq 0, \ X^x_t \leq X^y_t \quad \text{a.s.}
$$ 
Using this coupling $(X^x, X^y)$, we have
\begin{align*}
\mathcal{W}(\delta_x P_t, \delta_y P_t)
&\leq \E[|X^x_t - X^y_t|]\\
&\leq \E[X^y_t] - \E[X^x_t]\\
&\leq P_t (id) (y) - P_t (id) (x)\\
&\leq \sup_{z\in E} \left(P_t(id)\right)'(z) |x-y|.
\end{align*}
The bound for $\sup_{z\in E} \left(P_t(id)\right)'(z)$ can be found using the generator. Nevertheless, this proof did not give any information about the optimality. Our approach confirms the optimality of this coupling. This coupling is the same as \cite{TCP}. It favours the simultaneous jumps.
\end{Rq}


Now, we give an application of Lemma \ref{th:WC=0} which is a criterion for an exponential convergence when $\rho=0$.

\begin{theo}[Exponential decay when the curvature is null]
\label{th:cv-WC=0}
Under the same assumptions of the previous lemma and if $V \geq 0$, $Y$ is irreducible, there exist a compact set $K =[a,b]$ and a constant $\varepsilon>0$ such that
$$
\forall x \notin K, V(x) \geq \varepsilon,
$$
then there exist $\tilde{t}\geq0$ and $\kappa>0$ such that
$$
\forall t \geq \tilde{t}, \ \mathcal{W}(\mu P_t, \nu P_t) \leq e^{- \kappa t} \ \mathcal{W}(\mu, \nu ),
$$
for any probability measure $\mu,\nu$.
\end{theo}
Notice that, if there exist $x\in E$ such that $V(x)=0$ then the Wasserstein curvature is null.

\begin{proof}
The proof is adapted to \cite[Section 5.1]{MT06}. Let
$$
D(t,x)= \E_x\left[\exp\left( - \int_0^t V(Y_s) ds \right)\right],
$$
and $\bar{D}(t)=\sup_{x\in E} D(t,x)$. Here $\E_x [\cdot]$ stands for $\E[ \cdot | Y_0=x]$. It is easy to see that for all $x\in E$, $D(\cdot,x)$ and $\bar{D}$ are non increasing. Furthermore, if $D(t,x) =1$ and $Y_0=x$ then $V(Y_s)=0$ almost surely, for all $s\leq t$. Then, as $Y$ is irreducible, there exists $t_0>0$ such that we have
\begin{equation}
\label{eq:D<1}
\forall x\in E, \  \forall t> t_0, D(t,x)<1.
\end{equation}
Now, we begin to prove the existence of $t_1 \geq 0$ such that
$$
\forall t\geq t_1, \  \bar{D}(t) <1.
$$
If $x\notin K$ then we set $\tau= \inf\{t\geq 0 \ | \ Y_t \in K \}$. We have,
\begin{align*}
D(t,x)
&= \E_x\left[\1_{\tau < t} \exp\left( - \int_0^t V(Y_s) ds \right)\right]
+ \E_x\left[\1_{\tau \geq t} \exp\left( - \int_0^t V(Y_s) ds \right)\right]\\
&\leq \E_x\left[\1_{\tau < t} \exp\left( - \int_0^t V(Y_s) ds \right)\right] + e^{-\varepsilon t}.
\end{align*}
And
\begin{align*}
\E_x\left[\1_{\tau < t} \exp\left( - \int_0^t V(Y_s) ds \right)\right]
&\leq \  \E_x\left[ \1_{\tau < t} e^{- \varepsilon \tau } \E_x\left[ \exp\left( - \int_\tau^t V(Y_s) ds \right) \ | \ \mathcal{F}_\tau \right] \right] \\
&\leq \ \E_x\left[ \1_{\tau < t} e^{- \varepsilon \tau} \max_{c\in \{a,b\}}D(t - \tau, c) \right] \\
&\leq \ \E_x\left[ \1_{\tau < t/2} e^{- \varepsilon \tau}  \max_{c\in \{a,b\}}D(t - \tau, c) \right] \\
&+ \E_x\left[ \1_{t/2 \leq \tau < t} e^{- \varepsilon \tau}  \max_{c\in \{a,b\}}D(t - \tau, c) \right] \\
&\leq  \max_{c\in \{a,b\}} D(t/2, c) + e^{-\varepsilon t /2},
\end{align*}
where $\mathcal{F} = (\mathcal{F}_t)_{t\geq0}$ is the natural filtration associated to $Y$. Thus,
$$
\sup_{x \notin K} D(t,x) \leq  \max_{c\in \{a,b\}} D(t/2, c) + e^{-\varepsilon t} + e^{-\varepsilon t /2}.
$$
We deduce that $\limsup_{t \rightarrow +\infty} \sup_{x \notin  K} D(t,x) <1$ and the existence of $t_1$ such that for all $t\geq t_1$, 
$$\sup_{x \notin K} D(t,x) <1.$$ 
The Feynman-Kac semigroup is continuous on $K$. So, we deduce that $\bar{D}(t)<1$ for all $t\geq \tilde{t}=\max(t_0,t_1)$. Lemma \ref{th:WC=0} ends the proof. We can also use the argument to \cite{MT06} (which is very similar). That is, the Markov property ensures that
\begin{align*}
D(t+s,x)
&=\E_x\left[ \exp\left(-\int_0^t V(Y_u) du\right) \E_{Y_t}\left[ \exp\left(-\int_0^t V(Y_u) du\right) \right] \right]\\
&\leq\E_x\left[ \exp\left(-\int_0^t V(Y_u) du\right) \bar{D}(s) \right]= D(t,x)\bar{D}(s)\\
&\leq \bar{D}(t)\bar{D}(s),
\end{align*}
and then $\bar{D}(t+s) \leq \bar{D}(t)\bar{D}(s)$. For all $t>\tilde{t}$, there exists $n\in \N$ such that $t \geq n \tilde{t}$. Thus,  we have
\begin{align*}
\mathcal{W}(\mu P_t, \nu P_t) 
&\leq \bar{D}(t) \ \mathcal{W}(\mu, \nu ) \leq \bar{D}(n \tilde{t}) \ \mathcal{W}(\mu, \nu ) \\
&\leq \bar{D}(\tilde{t})^n \ \mathcal{W}(\mu, \nu ) \leq \exp\left(t \ \frac{\ln(\bar{D}(\tilde{t}))}{\tilde{t}} \right) \ \mathcal{W}(\mu, \nu ). 
\end{align*}

\begin{Rq}[A link with the quasi-stationary distribution (QSD) ]
If $V \geq 0$ then we have another representation of the gradient of $(P_t)_{t\geq0}$. Indeed we have
$$
\partial_x P_t f(x) = \E[ f'(Y_t) \1_{t<\tau} \ | \ Y_0=x],
$$
where $Y$ is a Markov process generated by $\mathcal{L}_S$ and $\tau$ verifies 
$$
\p(\tau>t \ | \ Y_s, s\leq t) = \exp \left( - \int_0^t V(Y_s) ds\right).
$$
If $Y$ admits a Yaglom limit, that is, there exists $\mu$ such that
$$
\lim_{t \rightarrow + \infty} \E[ f'(Y_t) \ | \ Y_0=x, t < \tau] = \int f d\mu,
$$
then we have $\p (\tau > t) \sim e^{-\theta t}$, for some $\theta>0$. And then, if $\pi$ is the invariant distribution of $X$,
\begin{align*}
\left|P_t f(x) - \int_E f d\pi \right|
&= \left| \int_E \int_{[x,y]} \E[ f'(Y_t) \ | \ Y_0=u] du \ d\pi(y) \right| \\
&\leq C e^{-\theta t}.
\end{align*}
\end{Rq}

\begin{Rq}[On the usage of an other gradient]
Another way to prove an exponential decay in the case of non positive curvature is to change the distance. Let $a$ be a positive and increasing function, the following mapping defines a distance:
$$
\forall x,y \in E, d_a(x,y) = |a(x) -a(y)|.
$$
Our proof is generalisable for this distance. It is enough to commute $\mathcal{L}$ and $\nabla_a= a \nabla$ instead of $\mathcal{L}$ and $\nabla$. This method is efficient for the $M/M/1$ in \cite{CJ10}. We have not followed this approach. Nevertheless it can improve the rate of convergence, it do not generalise the field of concerned processes; that are the stochastically monotonous processes.
\end{Rq}

\end{proof}

\section{Examples and applications}
\label{sect:Example}

In this section we develop several examples. In Subsection \ref{sect:exe-direct}, our main results are simply applied to some models with biological applications. In Subsection \ref{sect:Langevin}, we use our intertwining relation and an $h-$transform to obtain a rate of convergence for Kolmogorov-Langevin processes. In the section that follows, we use our main theorem to find a bound in total variation for some PDMP. We finish by several remarks for the example of the TCP process.

\subsection{Stochastic models for population dynamics}
\label{sect:exe-direct}

\subsubsection{Feller diffusion with multiplicative jumps: The Bansaye-Tran process}
Let us consider the process studied in \cite{Feller}. It lives on $E= \R_+$ and evolves according to a Feller diffusion; namely
\begin{equation*}
 d X_t=  \mathbf{g} X_t dt+  \sqrt{2 \mathbf{s} X_t}d B_t,
\end{equation*}

where $\mathbf{g}$ and $\mathbf{s}$ are two positive numbers. When it jumps from $x$, this new state is $Hx$, where the random variable $H$ is distributed according to a probability measure $\mathcal{H}$ on $[0,1]$. We assume that $H \egalloi 1-H$. This process models the rate of parasite in a cell population. The number of parasite grows in each cell and, sometimes, the cells divide. These two phenomena do not unwind in the same time scale. The parasites born and die faster than the cells divide. Thus, the rate of parasite is modelled by a Feller diffusion. This one can be understood as the limit of birth and death process. The jumps model the division of cell. In this setting, we have

\begin{coro}[Exponentially decreasing to $0$ when $r$ is decreasing]
If $r$ is decreasing and
\begin{align*}
\rho
&= \E[H] \left( \inf_{x\geq 0} r(x) -x \bar{r'}  \right) - \mathbf{g} > 0,
\end{align*}
where $\bar{r'} = \sup_{x\geq0} r'(x)$, then for any $t\geq0$,
$$
 \E[|X_t|] \leq e^{  -t \rho} \E[|X_0|].
$$
\end{coro}
\begin{proof}
Using theorem \ref{th:Intro-FK}, we deduce that Wasserstein curvature $\rho$ is positive. Furthermore, $\delta_0$ is invariant and $\mathcal{W}(\delta_x P_t, \delta_0) = P_t Id(x) = \E[X_t \ | \ X_0=x]$. 
\end{proof}

In \cite{Feller}, it is proved that if $r$ is monotonous then $X$ convergences, almost surely, to zero. They do not give an explicit bound for the convergence. Our corollary gives a (new) bound for the $L^1$-convergence. To compare, for instance, \cite[Proposition 3.1]{Feller} says

\begin{theo}[Extinction criterion when $r$ is constant] \label{th:BT} We have the following duality. \\
(i) If $ \mathbf{g} \leq - r \int_0^1 \log(h) \mathcal{H} (dh)$, then  $\p(\exists t >0, \ Y_t=0)=1$. \\
Moreover if $\mathbf{g} < - r \int_0^1 \log(h) \mathcal{H} (dh)$, 
\begin{equation}
\label{eq:BT-VT}
\exists \alpha>0,\, \forall x_0\geq 0,\, \exists c>0, \forall t\geq 0, \quad \p(X_t>0 \ | \ X_0 = x_0 )\leq ce^{-\alpha t}.
\end{equation}
(ii) If $\mathbf{g} > - r \int_0^1 \log(h) \mathcal{H} (dh)$, then  $\mathbb{P}(\forall t\geq 0, \  X_t>0)>0$.\\
Furthermore, for every $0\leq \alpha< \mathbf{g} + r \int \log(h) \mathcal{H}(dh)$,
\begin{equation*}
\p \left(\lim_{t\rightarrow +\infty} e^{-\alpha t} Y_t= \infty \right) = \left\{ \forall t, \  X_t>0  \right\} \quad \text{a.s.}
\end{equation*}
\end{theo}
The point \eqref{eq:BT-VT} can be written as
$$
d_{TV} \left( \delta_{x_0} P_t, \delta_0 \right) \leq ce^{-\alpha t}.
$$
And the second point (ii) implies that $$
\underline{\lim}_{t \rightarrow + \infty} d_{TV} \left( \delta_{x_0} P_t, \delta_0 \right) >0.
$$ 
But when $r$ is constant and the curvature $\rho$ is positive, we have $\mathbf{g} < - r \int_0^1 \log(h) \mathcal{H} (dh)$. Thus $X$ converges almost surely to $0$, and 
$$
\E[|X_t|] \leq e^{- \rho t} \E[|X_0|].
$$ 
More precisely, a rapidly calculation gives 
$$
\E[X_t]=e^{- \rho t} \E[X_0].
$$
So, if
$$
\int_0^1 - \log(h) \mathcal{H} (dh) > \frac{\mathbf{g}}{r} > \kappa= 1 - \int_0^1 h \mathcal{H} (dh),
$$
then
$$
\lim_{t \rightarrow + \infty} X_t = 0 \ \text{a.s.} \ \text{but} \ \lim_{t \rightarrow + \infty} \E[X_t] = + \infty.
$$

\subsubsection{Rate of convergence for branching measure-valued processes}

Let us consider a model of structured population. We observe a Markov process indexed by a supercritical continuous time Galton Watson tree.
Along the branches of the tree, the process evolves as a diffusion. The branching event
is nonlocal; namely the positions of the offspring are described by a random vector $(F_{j,K}(x, \Theta ))_{j\leq K}$. They depend on the position $x$ of the mother just before the branching event and on the number $K$ of offspring. The randomness of these positions is modelled via $\Theta$; it is a uniform variable on $[0,1]$. This process can be described with the following empirical measure:
$$
Z_t = \sum_{u\in V_t} \delta_{X^u_t},
$$ 
where $X^u_t$ lives in the branch $u$ at time $t$ and $V_t$ is the set of branch at time $t$. It was proved in \cite{GW} and \cite{C11} that
$$
\frac{1}{\E[Z_t(E)]} \E\left[\int_E f(x) Z_t(ds)\right] = \frac{1}{\E[\text{card}(V_t)]} \E\left[\sum_{u\in V_t} f(X^u_t) \right]  = \E\left[f(Y_t)\right],
$$
where $Y$ is generated by \eqref{eq:generateur}, with biased parameter. With this formula, we can deduce the long time behavior and the contraction properties of the mean measure. A similar formula holds when $r$ is not constant \cite{C11}. It will be interesting to capture the speed of convergence to $Z$ instead of $\E[Z]$. A first approach is given in \cite[Theorem 1.2]{C11}.

\subsection{Kolmogorov-Langevin processes}
\label{sect:Langevin}
Let us consider the process which verifies
$$
d X_t= \sqrt{2} d B_t - q'(X_s) ds,
$$
where $q$ is $C^\infty$ and $B$ is a standard Brownian motion. It is already known that, under suitable assumptions, this process converges to the Gibbs (or Boltzmann) measure $\pi(du)=e^{-q(u)} du/\mathcal{Z}$, where $\mathcal{Z}$ is a renormalizing constant. Theorem \ref{th:Intro-FK} shows that the curvature is equal to $\rho=\inf_{z\in \R} q''(z)$. It is trivial but we can hope an exponential decay to the invariant measure when $\rho<0$. In particular,  $\rho=0$ was studied in \cite[Section 5]{MT06}. We have

\begin{theo}[Wasserstein exponential ergodicity] Assume that
$$
\lim_{|x| \rightarrow + \infty} q''(x)= + \infty,
$$
then there exists $\lambda >0$ such that 
$$
\forall x,y \in \R, \ \exists C_{x,y} >0, \ \forall t\geq0,  \  \mathcal{W} \left( \delta_x P_t , \delta_y P_t \right) \leq C_{x,y}  e^{-\lambda t}. 
$$
Furthermore for all $x\in \R$ there exists $C_x$ such that
$$
\forall t \geq 0, \mathcal{W} \left( \pi , \delta_x P_t \right) \leq C_x  e^{-\lambda t}. 
$$
\end{theo}

Notice that $\lambda$ is "explicit". It is the first eigenvalue to the operator 
$$
f:\mapsto -f'' + f (q''/2 + q'^2/4).
$$
Or equivalently it is also the first eigenvalue to $f:\mapsto -f'' +q'(x)f'(x) $.
\begin{proof}
Let $(P_t)_{t\geq0}$ be the semigroup of $X$. By Lemma \ref{lem:Grad-est}, we have, for all $f$ smooth enough,
$$
\partial_x P_t f = S_t f',
$$
where $(S_t)_{t\geq0}$ is the Feynman-Kac semigroup generated by $\mathcal{L_S}$ defined by,
$$
\mathcal{L_S} f = f''(x) -q'(x)f'(x) - q''f(x).
$$
As said in Remark \ref{rq:h-transform}, we will do an $h-$transform. Let $H$ be the closure on $L^2$ of the operator defined by
$$
H f =-f'' + f \left( \frac{q''}{2} + \frac{q'^2}{4} \right).
$$
There exists a unique positive function $\varphi \in L^2 \cap C^\infty$ and a real number $\lambda>0$ such that $H\varphi = \lambda \varphi$. Indeed, It is known from \cite[Theorem 3.1 p. 57]{BS83} that if 
$$
\lim_{|x| \rightarrow + \infty} \frac{q''(x)}{2} + \frac{q'^2(x)}{4}= + \infty,
$$
then $H$ has a discrete spectrum. Furthermore, there exists a unique positive eigenvector. It corresponds to the smaller eigenvalue $\lambda>0$ and we denote it by $\varphi$ \cite[chapter 2]{BS83}.  The regularity of $\varphi$ comes from the regularity of $q$. Let $(Q_t)_{t\geq0}$ be defined for all smooth enough $f$ by
\begin{equation}
\label{eq:def-Q}
Q_t (f) =  \frac{e^{ \lambda t}}{\varphi} e^{-q/2} S_t(f \varphi e^{q/2}).
\end{equation}
We have
\begin{align*}
\partial_t Q_t (f)
&= \lambda \frac{e^{ \lambda t}}{\varphi} e^{-q/2} S_t(f \varphi e^{q/2}) + \frac{e^{ \lambda t}}{\varphi} e^{-q/2}  S_t( \mathcal{L_S}(f \varphi e^{q/2}))\\
&= \frac{e^{ \lambda t}}{\varphi} e^{-q/2} S_t(f e^{q/2} H \varphi) + \frac{e^{ \lambda t}}{\varphi} e^{-q/2}  S_t(  - e^{q/2} H(f \varphi))\\
&= \frac{e^{ \lambda t}}{\varphi} e^{-q/2} S_t( e^{q/2} \varphi G f) \\
& =Q_t(Gf),
\end{align*}
where
$$
G f= f'' + 2 \frac{\varphi'}{\varphi} f'.
$$
The relation \eqref{eq:def-Q} gives
$$
\partial_x P_t f(x) = e^{-\lambda t} e^{q(x)/2} \varphi(x) \E\left[ \frac{f'(Y_t)}{\varphi(Y_t)} e^{-q(Y_t)/2} \ | \ Y_0 =x \right],
$$ 
where $Y$ is a Kolmogorov-Langevin process generated by $G$ and starting from $x$. Thus, if $f$ is Lipschitz then,
$$
\left|\partial_x P_t f(x) \right| \leq  e^{- \inf_{z\in \R} q(z)/2} e^{-\lambda t} e^{q(x)/2} \varphi(x) Q_t \left( \frac{1}{\varphi} \right)(x).
$$
But
$$
G \frac{1}{\varphi} = \frac{- \varphi''}{\varphi^2} =\frac{1}{\varphi}\left( \lambda - \frac{q''}{2} - \frac{q'^2}{4} \right).
$$
So, by \cite[Theorem 6.1]{MT} with $V=1/\varphi$, we have
$$
\left| Q_t\left(\frac{1}{\varphi}\right)(x) -\tilde{\pi}\left(\frac{1}{\varphi} \right)  \right| \leq B \beta^t \left( 1 + \frac{1}{\varphi(x)} \right),
$$
for all $x\in \R$, where $\beta \in (0,1)$, $B\geq0$ and $\tilde{\pi}(dx)= e^{\int^x 2 \varphi'/\varphi} dx/\mathcal{Z}= \varphi(x)^2 dx/\mathcal{Z}$ ($\mathcal{Z}$ is a renormalizing constant). Notice that we also have
$$
\tilde{\pi}\left(\frac{1}{\varphi} \right) = \int_\R \varphi(x) dx < + \infty.
$$
Finally,
$$
\left|\partial_x P_t f(x) \right| \leq  e^{- \inf_{z\in \R} q(z)/2} e^{-\lambda t} e^{q(x)/2} \varphi(x) \left(\tilde{\pi}\left(\frac{1}{\varphi} \right) + B \beta^t \left( 1 + \frac{1}{\varphi(x)} \right) \right).
$$
Thus, the first inequality of the theorem holds with
$$
C(x,y)= C \times \sup_{z \in [x,y]} e^{q(z)/2} \left( 1 +  \varphi(z) \right) |x-y|.
$$
Furthermore the Cauchy-Schwarz inequality gives that for all $x\in \R$, we have $\int_\R C(x,y)\pi(dy) < +\infty$.
\end{proof}

\begin{Rq}[$h-$transform and Sch\"{o}dinger equation]
The transformation \eqref{eq:def-Q} is usual in the study of schr\"{o}dinger equation \cite{P95}. It has many applications in the study of absorbed processes (to estimate the law of the Q-process), of branching measures \cite{C11}, but to our knowledge our approach is new.
\end{Rq}

\begin{Rq}[Ornstein-Uhlenbeck process]
If $q(x) = \mu x^2/2$ then the assumptions of the theorem do not hold. But we can follow the proof step by step. The mapping $\varphi: x \mapsto e^{-\mu x/2}$ is an eigenvector of $H$ with respect to the eigenvalue $\mu$. So, we find that $G= L$ and
$$
\mathcal{W}\left( \delta_x P_t, \delta_y P_t \right) \leq e^{- \mu t}|x-y|.
$$
In fact, it easy to see that it is an equality. This example points the sharpness of our method. 
\end{Rq}

\subsection{Total variation decay of PDMP}

Let $X$ be a piecewise deterministic Markov process (PDMP) on $E\subset \R$; that is, it is generated by
$$
L f(x) = g(x) f'(x) + r(x) \left( \int_{E} f(y) K(x,dy) - f(x) \right),
$$
where $K$ is a Markov kernel and $g,r$ was described below. By Theorem \ref{th:Intro-FK}, if $X$ is stochastically monotonous then its Wasserstein curvature is given by
$$
\rho= \inf_{x\in E} \left( -g'(x) + r(x) \left( 1 - \partial_x \left(\int_E y K(x,dy)\right)  \right)+ r'(x) \int_E (y-x) K(y, dx)  \right).
$$
If $\rho$ is positive and $K$ "contracts in total variation" then we are able to prove an exponential decay in total variation distance. Let us recall that, for any probability measure $\mu_1,\mu_2$, the total variation distance is defined by
$$
 d_{TV} (\mu_1,\mu_2) =\inf \p(X_1 \neq X_2),
$$
where the infimun is taken over all couple $(X_1,X_2)$ such that $X_1,X_2$ are respectively distributed according to $\mu_1,\mu_2$.  Hereafter, we will use the following notations:
$$
\bar{g'}=  \sup_{z\in E} g'(z) \ \text{ and } \ \underline{r} = \inf_{z\in E} r(z).
$$
\begin{theo}[Total variation decay for monotonous PDMP]
\label{th:dtv-PDMP}
If the following assumptions hold:
\begin{itemize}
\item[i)] the Wasserstein curvature is lower bounded $\rho \geq \kappa>0$;
\item[ii)] $r$ is "sufficiently" lower bounded; namely $\underline{r} > 0 \wedge \bar{g'} $;
\item[iii)] there exists $C>0$ such that
$$
\forall x,y \in E, \ d_{TV} \left( K(x, \cdot), K(y, \cdot) \right) \leq C |x-y|;
$$
\end{itemize}
then
\begin{equation}
\label{eq:dtv-PDMP}
 d_{TV} \left( \mu P_t, \nu P_t \right)  \leq K_{\mu,\nu} e^{- \theta t},
\end{equation}
where $\theta= \frac{\kappa \underline{r}}{\kappa + \underline{r}}$ and 
\begin{align*}
K_{\mu,\nu}&= \left( \kappa \left( C + \frac{\bar{r'}}{\underline{r}} \right)\frac{\underline{r}}{\underline{r} - \bar{g'}}  \mathcal{W}\left( \mu, \nu \right)\right)^{\frac{\underline{r}}{\underline{r}+ \kappa}}\\
&+ \kappa^{\frac{\kappa}{\underline{r}+ \kappa}}\left( \left( C + \frac{\bar{r'}}{\underline{r}} \right)\frac{\underline{r}}{\underline{r} - \bar{g'}}  \mathcal{W}\left( \mu, \nu \right)\right)^{\frac{\underline{r} + 2 \kappa}{\underline{r}+ \kappa}}.
\end{align*}
\end{theo}
Notice that, if X do not jump before $t$ then its trajectory is deterministic. If $r$ is constant then we have $d_{TV} \left( \delta_x P_t, \delta_y P_t \right)  \geq e^{rt}$. The point ii) provides that the first jump comes before a time  which is exponentially distributed. The point iii) is the main assumption of this theorem. It means that we can stick the two marginals of a coupling if they are close.

\begin{Rq}[A storage model]
Let $(X_t)_{t\geq0}$ be the Markov process, on $E= \R_+^*$, generated by
$$ 
L f(x) =  - \mathbf{g} x f'(x) + r(x) \left( \int_{\R_+} f(x+u) \lambda e^{- \lambda u} du - f(x) \right).
$$
Here $\mathbf{g}, \lambda>0, r>0$. It models a stock. The current stock decreases exponentially, and increases at inhomogeneous random times by a random amount (distributed following an exponential variable). We deduce directly, from our main theorem, that if $r$ is increasing then for any $x,y\geq0$,
$$
\forall t\geq0, \mathcal{W} \left(\delta_x P_t, \delta_y P_t\right) \leq e^{ -( \mathbf{g} + \frac{1}{\lambda} \inf_{x\geq0} r'(x)) t} |x-y|.
$$
On this example, the constants of the previous theorem are
\begin{align*}
K_{\mu,\nu} 
&=\left( \rho \left( \frac{1}{\lambda} + \frac{\bar{r'}}{r(0)} \right)\frac{r(0)}{r(0) - \mathbf{g}}  \mathcal{W}\left( \mu, \nu \right)\right)^{\frac{r(0)}{r(0)+ \rho}}\\
&+ \rho^{\frac{\rho}{r(0)+ \rho}}\left( \left( \frac{1}{\lambda} + \frac{\bar{r'}}{r(0)} \right)\frac{r(0)}{r(0) - \mathbf{g}}  \mathcal{W}\left( \mu, \nu \right)\right)^{\frac{r(0) + 2 \rho}{r(0)+ \rho}}.
\end{align*}
and
$$
\theta = \frac{\rho r(0)}{\rho + r(0)} .
$$
For instance, if $r$ is constant, we have
$$
 d_{TV} \left( \mu P_t, \nu P_t \right)  \leq e^{- t \frac{\mathbf{g}r}{\mathbf{g} + r} } K_{\mu,\nu}.
$$
This rate is not optimal \cite{BCGMZ}. Our approach is similar to \cite{BCGMZ}, we build a coupling such that the components are closer on $[0,s]$ and we change the coupling to stick the components on $[t-s,t]$. In \cite{BCGMZ}, the time $s$ is random while in our proof, it is deterministic. Here, $s$ is not random because when $r$ is not constant the countable process associated at the jumps is not a Poisson process.
\end{Rq}
The proof is based on the following lemma, proved via coupling argument:

\begin{lem}[Local total variation estimate]
Let $x > y$ and $t\geq 0$, under the same assumption of Theorem \ref{th:dtv-PDMP}, we have
$$
d_{TV} (\delta_x P_t, \delta_y P_t)
\leq e^{- t \underline{r}} +  |x-y|\left( C + \frac{\bar{r'}}{\underline{r}} \right) \frac{\underline{r}}{\underline{r} - \bar{g'}} .
$$
Here $\bar{r'} = \sup_{z\in E} r'(z)$.
\label{lem:dtv}
\end{lem}

\begin{proof}[Proof of Theorem \ref{th:dtv-PDMP}]
The previous expression is equivalent to
$$
d_{TV} \left(\mu P_t, \nu P_t \right) \leq e^{- t \underline{r}} + \left( C + \frac{\bar{r'}}{\underline{r}} \right) \frac{\underline{r}}{\underline{r} - \bar{g'}} \ \mathcal{W} (\mu,\nu),
$$
for any $\mu$ and $\nu$ which have a first moment. As $\rho>0$, we deduce that, for all $s\leq t$,
\begin{align*}
d_{TV} \left( \mu P_t, \nu P_t \right) 
&= d_{TV} \left( (\mu P_s) P_{t-s}, (\nu P_s) P_{t-s} \right)\\ 
&\leq  e^{- (t-s) \underline{r}} + e^{- \rho s} \left( C + \frac{\bar{r'}}{\underline{r}} \right)\frac{\underline{r}}{\underline{r} - \bar{g'}}  \mathcal{W}\left( \mu, \nu \right).
\end{align*}
And thus, $ d_{TV} \left( \mu P_t, \nu P_t \right)  \leq K_{\mu,\nu} e^{- \theta t}$
\end{proof}

\begin{proof}[Proof of Lemma \ref{lem:dtv}]
Let us consider the coupling $(X,Y)$, starting from $(x,y)$, and generated by
\begin{align*}
Gf(x,y)&= g(x) \partial_x f(x,y) + g(y) \partial_y f(x,y)\\
&+ r(x) \wedge r(y)  \left( \int_{E \times E} f(u, v) \mathbb{K}((x,y), d(u,v)) -  f(x,y) \right)\\
&+ (r(x)-r(y))^+ \left( \int_{E} f(u, y) K(x,du) -  f(x,y) \right)\\
&+ (r(y)-r(x))^+ \left( \int_{E} f(x, u) K(y,du) -  f(x,y) \right).
\end{align*}
$\mathbb{K}$  is choose such that
$$
\int_{E\times E} \1_{u \neq v} \mathbb{K}((x,y), d(u,v)) = d_{TV} \left( K(x, \cdot), K(y, \cdot) \right).
$$
The dynamics of this coupling is as follow.
\begin{itemize}
\item It start from $(X_0,Y_0) =(x,y)$ and for all $t < T$, $X'_t = g(X_t)$ and $Y'_t = g(Y_t)$.
\item The time $T$ verifies
$$
\p(T>t)= \exp\left(- \int_0^t r(x_s) \vee r(y_s) ds \right),
$$
where $x_0=x, \ y_0=y$ and $x'_t=g(x_t), \ y'_t=g(y_t)$.
\item At time $T$, we toss a coin $B$ such that
$$
\p(B=0 \ | \ T ) = \frac{r(x_T) \wedge r(y_T) }{r(x_T) \vee r(y_T)} \ \text{and} \ \p(B=1 \ | \ T ) = \frac{|r(x_T) - r(y_T)|}{r(x_T) \vee r(y_T)}.
$$
If $B=0$ then the two trajectories jump simultaneously and if $B=1$ only one component jumps.
\item If the two trajectories jump in the same time then we stick them.
\item We repeat these steps starting from $(X_T,Y_T)$.
\end{itemize}

We would like to stick the trajectories at the first jump (to stick them before is impossible), and so maximise the quantity $\p\left( X_t =Y_t \right)$ i.e.
$$
\p\left( X_t = Y_t \right) \geq \p\left( X_{T} = Y_{T}, t \geq T, B=0  \right).
$$
And as we have
$$
| x_t - y_t | = x_t-y_t \leq e^{T \sup_{z \in E} g'(z)} |x-y|,
$$
we deduce,
\begin{align*}
\p\left( X_{T} = Y_{T}, t \geq T, B=0  \right) 
\geq  &\E\left[  \1_{t \geq T, B=0} \p\left( X_{T} = Y_{T} \ |\ (T,B) \right) \right]\\
\geq  &\E\left[  \1_{t \geq T, B=0} \left( 1 - C|x-y| e^{T \bar{g'}} \right) \right]\\
\geq  &\E\left[  \1_{t \geq T} \left( 1 - C|x-y| e^{T \bar{g'}} \right) \p\left(  B=0 \ | \ T \right) \right]\\
\geq  &\E\left[  \1_{t \geq T} \left( 1 - C|x-y| e^{T \bar{g'}} \right) \left(  1 - \frac{\bar{r'}}{\underline{r}}  |x-y| e^{T \bar{g'}} \right) \right].\\
\end{align*}
Finally, as we can upper bounded $T$ by an exponential variable $E$ with parameter $\underline{r}$, we conclude that
\begin{align*}
d_{TV} (\delta_x P_t, \delta_y P_t)
&\leq \p( X_t \neq Y_t)\\
&\leq \E\left[ \1_{t<T} + \1_{t \geq T} |x-y| e^{T \bar{g'}} \left( C + \frac{\bar{r'}}{\underline{r}} \right) \right] \\
&\leq \p(t<E) +  |x-y|\left( C + \frac{\bar{r'}}{\underline{r}} \right) \E\left[ \1_{t \geq T} e^{T \bar{g'}} \right] \\
&\leq e^{- t \underline{r}} +  |x-y|\left( C + \frac{\bar{r'}}{\underline{r}} \right) \frac{\underline{r}}{\underline{r} - \bar{g'}}.
\end{align*}
\end{proof}

We can give similar results when the curvature is null:



\begin{coro}[A bound when $\rho=0$]
Assume that ii) and iii) hold, $X$ is stochastically monotonous and irreducible, $\rho = 0$ and there exist $[a,b] \subset E$ and $\varepsilon>0$ such that for all $x\notin [a,b]$ 
$$
\left( -g'(x) + r(x) \left( 1 -\partial_x \int_E y K(x,dy) \right)+ r'(x) \int_E y-x K(y, dx)  \right) \geq \varepsilon.
$$
Then there exist $K>0$ and $\theta>0$ such that
$$
d_{TV} (\mu P_t, \nu P_t) \leq K (1 + \mathcal{W}(\mu, \nu)) e^{- \theta t}.
$$
For any starting distribution $\mu, \nu$.
\end{coro}
\begin{proof}
It is a direct application of Theorem \ref{th:cv-WC=0}.
\end{proof}

\begin{Rq}[Total variation decay and jump-diffusions]
We can prove a similar result if the process have a diffusive part. Nevertheless, if $X$ diffuses then the convergence will be faster.
\end{Rq}

\begin{Rq}[A better criterion?]
The main assumption of this theorem is a contraction of the kernel in total variation (point iii)). This assumption is natural but not always interesting in the applications. For instance, we can have
$$
K(x,dy)=\delta_{F(x)},
$$
for some $F$. It is the case for the TCP process \cite{BCGMZ}.
\end{Rq}

\subsection{TCP window size process}

\subsubsection{The continuous time process}
Now, we consider a process which represent the TCP congestion. This Markov process, $(X_t)_{t\geq 0}$ is generated, for any smooth enough function $f$ and $x\geq0$, by
\begin{equation}
\label{eq:genTCP}
\mathcal{L} f(x) = f'(x) + r(x) \left( \int_0^1 f(hx) \mathcal{H}(dh) - f(x) \right).
\end{equation}
The invariant probability measure is explicit when $r(x)= \mathbf{r} x^\alpha$. It is explain in the following. Our main result gives
\begin{coro}[Wasserstein curvature]
If $r$ is non increasing, we have that $\rho$, defined at \eqref{eq:contract_lip}, verifies
\begin{align}
\label{eq:curv-TCP} 
\rho= \left( 1 - \int_0^1 h \mathcal{H}(dh) \right) \inf_{x\geq 0} \left( r(x) - x r'(x) \right).
\end{align}
\end{coro}
In \cite{TCP}, it was proved that
$$
r(x) =x +a \Rightarrow \rho \geq \frac{a}{2}.
$$
But when the process is near to $0$, the jump rate is $a$ so you have $\rho=a/2$. This bound is the same to \eqref{eq:curv-TCP}. It seems to be a coincidence. A recent work \cite{BCGMZ} prove that, nevertheless the curvature is null when $r(x)=x$, this process converges exponentially to its invariant distribution in Wasserstein and total variation distance. If $r$ is non increasing and $\rho=0$, Theorem \ref{th:cv-WC=0} tell us that we can have an exponential decay in Wasserstein distance.

\begin{Rq}[Poincar\'e or log-Sobolev inequality and the Bakry-Emery criterion]
\label{eq:TCP-BE}
Our approach is based on a commutation formula as well as the Bakry-Emery calculus \cite{BE85}. Nevertheless, in general, our processes do not verify a Poincar\'{e} or a log-Sobolev inequality. That is, for the first one,
$$
\lambda \text{Var}_{\pi} f \leq \int \Gamma f d\pi,
$$
where $\Gamma f = \frac{1}{2} L (f^2) - f Lf$ and $\lambda>0$. Indeed, in the case of the TCP window size, we have
$$
\Gamma f(x) = r(x) \left( \int_0^1 f( h x) \mathcal{H} (dh) - f(x) \right)^2.
$$
And if $\mathcal{H} = \delta_{1/2}$, we easily construct a lot of functions $f$ such that $\Gamma f = 0$ (see \cite{LP}). Thus, we have an example where the process have a positive Wasserstein curvature and which do not verify a Poincar\'e inequality. Note that $X$ is not reversible, thus it do not contradict Theorem \ref{th:Intro-spectralgap}.
\end{Rq}

\subsubsection{The embedded chain} Let $(\hat{X}_n)_{n \geq 0}$ be the embedded chain of the TCP process. That is defined by
\begin{equation}
\label{def:chaine}
\hat{X}_n = X_{T_n} \ \text{ where } \ T_n = \inf\{ t > T_{n-1} \ | \ X_{t+} \neq X_{t-} \} \ \text{ for $n \geq 1$ and } \ T_0 =0.  
\end{equation}
This Markov chain is often easier to study than the continuous time process. For instance, if $r(x)= ax^{\alpha}$, it is easy to see that 
$$
R(\hat{X}_{T_{n+1}})= R(H_n ( R^{-1}(E_n  + R(\hat{X}_{T_n})) ))=  H_n^{\alpha +1} ( E_n  + R(\hat{X}_{T_n})),
$$
where $R$ is the antiderivative of $r$ and $H_n =X_{T_n}/X_{T_n -}$. This autoregressive relation gives the ergodicity. Furthermore, the limiting random variable $\hat{X}_{\infty}$ verifies
$$
R(\hat{X}_{T_{\infty}}) \egalloi   H_1^{\alpha +1} ( E_1  + R(\hat{X}_{\infty})).
$$
Now, using \cite[Proposition 5]{GRZ}, we deduce that $\hat{X}_\infty$ have a density given by
$$
x \mapsto \frac{1}{ \prod_{n \geq 1} (1 - \mathbf{h}^{(\alpha +1)n})} \sum_{n \geq 0} \prod_{k=1}^n \frac{\mathbf{h}^{-(\alpha +1)(n+1)}}{1 - \mathbf{h}^{-(\alpha +1)k}}  a x^\alpha e^{- \mathbf{h}^{-(\alpha +1)(n+1)} a(\alpha +1)^{-1} x^{(\alpha+1)}}.
$$
Now, applying \cite[Theorem 34.31]{D93}, we can deduce the invariant law of the continuous time process. This result generalises \cite{GRZ}, but it is already known via others techniques \cite{vLLO}. For this Markov chain, we arrive to bound all Wasserstein distance. Recall that for every $p\geq1$, the $\mathcal{W}^{(p)}$ Wasserstein distance, between two laws $\mu_1$ and $\mu_2$ on $E$ with finite $p^{\text{th}}$ moment, is defined by
$$
\mathcal{W}^{(p)} (\mu_1,\mu_2) = \inf_{\{X \sim \mu_1, Y \sim \mu_2 \}} \left( \E\left[|X-Y|^p\right] \right)^{1/p},
$$
where the infimum runs over all coupling of $\mu_1$ and $\mu_2$. We have

\begin{theo}[Wasserstein exponential ergodicity for the embedded chain]
  Assume that $\mathcal{L}(X_0)$ and $\mathcal{L}(Y_0)$ have finite $p^\text{th}$ moment for some real $p\geq 1$ and $r$ is increasing. Let $\hat X$ and $\hat Y$ be the embedded chains of $X$ and $Y$. Then, for any
  $n\geq 0$, with a random variable $H\sim \mathcal{H}$,
  $$
  \mathcal{W}^{(p)}(\mathcal{L}(\hat X_n),\mathcal{L}(\hat Y_n)) \leq \E(H^p)^{n/p} \mathcal{W}_p(\mathcal{L}(X_0),\mathcal{L}(Y_0)).
  $$
  In particular, if $\hat\pi$ is the invariant law of $\hat X$ then
  $$
  \mathcal{W}^{(p)}(\mathcal{L}(\hat X_n),\hat\pi) \leq \E(H^p)^{n/p} \mathcal{W}_p(\mathcal{L}(X_0),\hat\pi).
  $$
\end{theo}
This result generalises \cite[Theorem 2.1]{TCP} but the proof is exactly the same and we give it for sake of completeness.
\begin{proof}
  It is sufficient to provide a good coupling. Let $x\geq0$ and
  $y\geq0$ be two non-negative real numbers, and let ${(E_n)}_{n\geq1}$ and ${(H_n)}_{n\geq 1}$ be two independent sequences of
  i.i.d.\ random variables with respective laws the exponential law of
  unit mean and the law $\mathcal{H}$. Let $\hat
  X$ and $\hat Y$ be the discrete time Markov chains on $[0,\infty)$
  defined by
  \begin{align*}
    \hat X_0&=x\quad\text{and}\quad %
    \hat X_{n+1}=H_{n+1}R^{-1}( R(\hat X_n) +E_{n+1})%
    \quad\text{for any $n\geq0$}\\
    \hat Y_0&=y\quad\text{and}\quad %
    \hat Y_{n+1}=H_{n+1}R^{-1}(R(\hat Y_n)+E_{n+1})%
    \quad\text{for any $n\geq0$}.
  \end{align*}
  The law of $\hat X$ (respectively $\hat Y$) is the law of the embedded chain of a process
  generated by $L$ and starting from $x$ (respectively $y$). Now, let $a$ be a non-negative number, if $\varphi_a : x\mapsto R^{-1} (a +R(x))$ then
\begin{equation}
\label{eq:preuve-rdecroit}
  \varphi_a'(x) = \frac{r(x)}{r\left(R^{-1}(a+R(x))\right)} \leq 1 \Rightarrow |\varphi_a(x) - \varphi_a(y)|\leq |x-y|.
\end{equation}
 And we get
  \begin{align*}
    \forall p\geq 1, \ \E[|\hat X_{n+1}-\hat Y_{n+1}|^p]&=
    \E[H_{n+1}^p|\varphi_{E_{n+1}}(\hat X_n)-\varphi_{E_{n+1}}(\hat Y_n)|^p]\\
    &\leq \E[H_{n+1}^p|\hat X_n-\hat Y_n|^p]
    =\E[H_{n+1}^p]\E[|\hat X_n-\hat Y_n|^p].
  \end{align*}
  A straightforward recurrence leads to
  $$
  \E[|\hat X_{n}-\hat Y_{n}|^p] \leq \E[H_1^p]^n|x-y|^p.
  $$
  It gives the desired inequality when the initial laws are Dirac masses.
  The general case follows by integrating this inequality with respect to
  couplings of the initial laws.
\end{proof}

This theorem gives a bound of the coarse Ricci curvature \cite{O10}, which is the discrete time equivalent of the Wasserstein curvature. We can compare the curvature of this Markov chain and its continuous time equivalent.

\subsection{Integral of L\'evy processes with respect to a Brownian motion}

Let us consider a fragmentation process i.e. $\int_0^1 f(F(x,\theta)) d\theta = \int f(h x) \mathcal{H}(dh)$ where $\mathcal{H}$ is a probability measure on $[0,1]$. As mentioned below, $X$ is solution to
\begin{align*}
X_t= X_0 &+ \int_0^t g(X_s) ds + \int_0^t \sigma(X_s) d B_s \\
&- \int_0^t \int_{E\ \times [0,1]} \1_{\{u \leq r \}}  \theta X_{s-} Q(ds,du,d\theta). \nonumber
\end{align*}
The jump times $T_1,...$ are distributed following a Poisson process $\mathcal{N}_t$. Between these times, the process evolves like a diffusion. At these times, we have $X_{T_j} = H_j X_{T_j -}$, where $(H_j)$ is a i.i.d. sequence of law $\mathcal{H}$. we assume that $H_1 \in (0,1)$ almost surely. If you take the logarithm of $X$ then the multiplicative jumps become additive jumps. Then, as the jump times are Poissonian, we can obtain a continuous process by renormalising our process with a L\'{e}vy process. Formally, let $(L_t)$ be the L\'{e}vy process defined by
$$
L_t = - \int_{\R_{+} \ \times [0,1]} \1_{ \{u \leq r \}} \ln(h) Q(ds,du,dh) = - \ln \left( \prod_{j=1}^{\mathcal{N}_t} H_j \right),
$$
and let $\bar{X}$ be the continuous process defined by $\bar{X}_t = X_t e^{L_t}$. We have
\begin{lem}[Stochastic differential equation for $\bar{X}$ and $X$]
For any $t\geq0$,
$$
\bar{X}_t = X_0 + \int_0^t e^{L_s} g(X_s) ds + \int_0^t e^{L_s} \sigma(X_s) d B_s 
$$
and,
$$
X_t = X_0 e^{-L_t} + \int_0^t e^{-L_s} g(X_{t-s}) ds + \int_0^t e^{-L_s} \sigma(X_{t-s}) d B_s.
$$
\end{lem}
\begin{proof}
All the stochastic integrals that we write are well defined as local martingales. Using It\^{o}'s formula with jumps \cite[Theorem 5.1, p.67]{IW} (see also \cite[Theorem 4.57, p.57]{JS}) and since $\bar{X}$ is continuous, we get
\begin{align*}
\bar{X}_t
&= X_0 + \int_0^t e^{L_s} \left( g(X_s) ds + \sigma (X_s) d B_s \right)\\
&+ \int_0^t X_{s-} e^{L_{s-}} - X_{s} e^{L_{s}} \1_{\{u \leq r\}} Q(ds,du,dh)\\
&= X_0 + \int_0^t e^{L_s} \left( g(X_s) ds + \sigma (X_s) d B_s \right).
\end{align*}
Then, we deduce,
\begin{align*}
X_t
&= X_0 e^{-L_t} + \int_0^t e^{L_s-L_t} \left( g(X_s) ds + \sigma (X_s) d B_s \right)\\
&= X_0 e^{-L_t} + \int_0^t e^{- L_{t-s}} \left( g(X_s) ds + \sigma (X_s) d B_s \right)\\
&= X_0 e^{-L_t} + \int_0^t e^{-L_s} g(X_{t-s}) ds + \int_0^t e^{-L_s} \sigma(X_{t-s}) d B_s.
\end{align*}
\end{proof}
This lemma is a generalisation of the relation of \cite[Lemma 3.2]{Feller} and \cite[Section 6]{lcst}. In \cite{Feller}, it is a preliminary for the proof of Theorem \ref{th:BT}. In \cite{lcst}, they deduce that when $g$ is constant, and $\sigma=0$, we have
$$
\lim_{t \rightarrow + \infty} X_t= g \int_0^{+\infty} e^{-L_s} ds.
$$
The behavior of the right hand side was studied in \cite{BY} and \cite{CPY} for general L\'{e}vy processes. We give another application. Let $Y$ be defined, for all $t\geq0$, by
$$
 Y_t = \int_0^t e^{-L_s} dB_s,
$$
where $L_t= \sum_{k=1}^{N_t} \ln(H_j)$ is independent from the Brownian motion $B$, and $(H_j)_{j\geq0}$ are i.i.d. and distributed according to $\mathcal{H}$. Let us define for all $n \in \N$, $\kappa_n = 1 - \int_0^1 h \mathcal{H}(dh)$, we have

\begin{theo}[Long time behavior of Integral of compound Poisson process with respect to an independent Brownian motion]
 $Y$ converges in law to a measure $\pi$ such that
 $$
 \int x^n \pi(dx)= \frac{n!}{r^{n/2} \prod_{k=1}^{n/2} \kappa_{2k}}
 $$
 for all $n \in 2 \mathbb{Z}$, and 
 $$
 \int x^n \pi(dx)= \frac{n!}{r^{(n+1)/2} \prod_{k=1}^{(n+1)/2} \kappa_{2k-1}}
 $$
 otherwise. Furthermore all its moments converges and for all $t\geq0$,
 $$
 \mathcal{W} \left( \mathcal{L} (Y_t), \pi \right) \leq e^{-\kappa_1 t} \mathcal{W} (\delta_0, \pi).
 $$
\end{theo}
 \begin{proof}
 By the previous lemma, we know that $Y$ is generated by (\ref{eq:generateur}), where, $g=0,\sigma=1$, $r$ is constant and $Y_0=0$. This process is positively curved thus it admits a unique invariant probability measure and converges exponentially to it. Furthermore, applying the generator on the functions $\alpha_n :x \mapsto x^n$  gives the moments of $\pi$ (we can use the Carleman criterion to prove that $Y$ converges also to a measure with this moment).

 \end{proof}



\textbf{Acknowledgements:} This paper was improved by some discussions with F. Malrieu and D. Chafa\"{i}. The author is grateful to them.

\bibliographystyle{amsalpha}
\bibliography{ref}

\end{document}